\documentclass[letterpaper]{amsart}
\usepackage{tikz-cd}
\usepackage[utf8]{inputenc}
\usepackage{amsmath,amssymb,amsthm}
\usepackage{hyperref}
\theoremstyle{plain} 
\newtheorem{theorem}{Theorem}
\newtheorem*{theorem*}{Theorem}
\newtheorem{prop}[theorem]{Proposition}
\newtheorem{lemma}[theorem]{Lemma}
\newtheorem{coro}[theorem]{Corollary}
\newtheorem{conj}[theorem]{Conjecture}
\theoremstyle{definition} \newtheorem{definition}{Definition}
\theoremstyle{definition}  
\theoremstyle{remark} \newtheorem*{remark}{Remark}
\theoremstyle{definition} \newtheorem*{ack}{Acknowledgment}
\author{Chuanhao Wei}
\title{Fibration of log-general type space over quasi-abelian varieties}
\begin{document}
\begin{abstract}
We show that there exists no smooth fibration of a smooth complex quasi-projective variety of log-general type over a quasi-abelian variety. The proof uses M. Popa and C. Schnell's construction of Higgs bundle.
\end{abstract}
\maketitle
\let\thefootnote\relax
\let\thefootnote\relax
\section{Introduction}
In \cite{ps2}, Popa and Schnell use the vanishing theory of Saito's Hodge modules \cite{ps1} to prove the following
\begin{theorem}\label{T:vgt}
For any projective smooth variety of general type, there exists no non-vanishing global holomorphic one-form on it.
\end{theorem}
As an easy corollary to this theorem, we get
\begin{coro}
There exists no smooth fibration of a smooth projective variety of general type over an abelian variety.
\end{coro}
To generalize the theorem to a log-version, we propose the following 
\begin{conj}
For any projective log-smooth pair $(X,D)$ of log-general type, there exists no non-vanishing global holomorphic log-one-form.
\end{conj}
According to \cite{s.i}, see also Section 3, in particular Proposition \ref{prop:globallog}, to prove this conjecture, it is natural to consider a morphism $f: X\to P^{r,d}$, and $f^{-1}L\subset D$, where $P^{r,d}$ is a natural compactification of a quasi-abelian variety $T^{r,d}$ (in the sense of Iitaka), with the boundary divisor $L$. Then we need to show that for any log-one form on $(X,D)$ which is of the form $f^*\theta$, where $\theta\in H^0(\Omega^1_{P^{r,d}}(log\ L))$, its zero locus is non-empty. When $r=0$ and $D=0$, this implies Theorem \ref{T:vgt}. However, the general case of the above conjecture is still open.

The difficulty of proving this conjecture using the method that appeared in \cite{ps2} is to better understand the relation between log-forms and Saito's Mixed Hodge modules. In this paper, we give a very simple proof of the following two weak versions of the conjecture.
\begin{theorem}
Let $(X, D)$ be a log smooth pair with log canonical bundle 
$\omega_X(D)$ which contains an ample line bundle. Then, for any global log-one-form $\theta\in H^0(\Omega^1_X(log\  D))$, the zero locus of $\theta$ must be non-empty.
\end{theorem}
\begin{theorem}\label{T:flgt}
Let $(X,D)$ be a projective log-smooth pair of log general type. Assume that we have a surjective projective morphism $f: X\setminus D \to T^{r,d}$, where $T^{r,d}$ is a quasi-abelian variety.  Then, $f$ is not smooth.
\end{theorem}

Given a smooth quasi-projective variety $V$, we say that $(X,D)$ gives a smooth compactification of $V$, if $(X,D)$ is log-smooth, i.e. $X$ is smooth and $D$ is a reduced divisor over $X$ with simply normal crossing support, and $X\setminus D=V$. Since the logarithmic Kodaira dimension does not depend on the smooth compactification, Theorem \ref{T:flgt} is only about $X\setminus D$.

When $d=0$ and $r=1$, Theorem \ref{T:flgt} is just Theorem 0.3 in \cite{vz}, and in fact the proof in this case is very similar to the one that appeared in Viehweg and Zuo's work. Note that this is a corollary to Conjecture 3 as explained in Section 4. 
\begin{ack}
The author would like to thank C. Hacon for suggesting this topic and for useful discussions. The author is grateful to M. Popa and C. Schnell for answering his questions. The author was partially supported by DMS-1300750, DMS-1265285 and a grant from the Simons Foundation, Award Number 256202.
\end{ack}

\section{Proof of Theorem 4}
\begin{proof}[Proof of Theorem 4]
Assume $\theta\in H^0(\Omega_X(log\  D))$ is a global log-one-form that vanishes nowhere. We have the following exact Koszul complex defined by $\theta$:
$$0\rightarrow\mathcal{O}_X\xrightarrow {\wedge \theta}\Omega^1_X(log\  D) \xrightarrow {\wedge \theta}...\xrightarrow {\wedge \theta}\Omega^n_X(log\  D)\rightarrow 0.$$
Let $\mathcal{E}^{0}=0,\mathcal{E}^{n}=\Omega^n_X(log\  D)$, and 
$$\mathcal{E}^i=ker(\wedge\theta): \Omega_x^i(log\  D)\to\Omega_X^{i+1}(log\  D),$$
for $i=1,...,n-1$.

Since $\omega_X(D)$ contains an ample line bundle, we can write 
$$\omega_X(D)=\mathcal{O}_X(L+E),$$
for some ample divisor $L$ and some effective (possibly trivial) divisor $E$ on $X$.

Twisting the Koszul complex above by $\mathcal{O}_X(L-D)$, we have the following exact sequence:
$$0\rightarrow\mathcal{O}_X(L-D)\xrightarrow {\wedge \theta}\Omega^1_X(log\  D)(L-D) \xrightarrow {\wedge \theta}...\xrightarrow {\wedge \theta}\Omega^n_X(log\  D)(L-D)\rightarrow 0.$$
This exact sequence can be decomposed into the following short exact sequences:
$$0\to \mathcal{E}^i(L-D)\to \Omega_X^i(log\  D)(L-D)\to \mathcal{E}^{i+1}(L-D)\to 0.$$
By Akizuki-Kodaira-Nakano vanishing \cite[6.4 Corollary]{ev}, we have 
$$H^{n-i+1}(\Omega_X^i(log\  D)(L-D))=0.$$
Hence the natural differential map associated to the short exact sequences
$$H^{n-i}(\mathcal{E}^{i+1}(L-D))\to H^{n-i+1}(\mathcal{E}^{i}(L-D))$$
is surjective.
Note that $H^1(\mathcal{E}^{n}(L-D))=H^1(\Omega^n_X(log\  D)(L-D))=0.$
Since $H^1(\mathcal{E}^n(L-D))\to H^n(\mathcal{E}^1(L-D))$
is surjective, we also have that
$$H^n(\mathcal{E}^{1}(L-D))=0.$$
However, since 
$\mathcal{E}^{1}=\mathcal{O}_X,$
and
$\mathcal{O}_X(L+E)=\omega_X(D),$
we have
$$\mathcal{E}^{1}(L-D)=\omega_X(-E).$$
Hence, by Serre duality, we have 
$$H^0(\mathcal{O}_X(E))\simeq H^n(\omega_X(-E))^\vee=0$$
which contradicts the fact that $E$ is effective.
\end{proof}
\begin{remark}
The proof is based on the arguments that appears  in \cite{hk}.
\end{remark}
\section{quasi-abelian varieties}
In this section, we recall the definition of quasi-abelian varieties in the sense of Iitaka, \cite{s.i}, and recall a couple of propositions that will be later used. We refer to Fujino's survey article \cite{o.f} for details.
\begin{definition}
$T^{r,d}$ is an quasi-abelian variety (in the sense of Iitaka), if it is an extension of a $d$-dimensional abelian variety $A^d$ by algebraic tori $\mathbb{G}_m^{r}$, i.e. it is a connected commutative algebraic group which has the following Chevalley decomposition
$$1\to \mathbb{G}_m^r\to T^{r,d}\to A^d\to 1.$$
In particular, $T^{r,d}$ is a principal $\mathbb{G}_m^r$-bundle over $A^d$.
\end{definition}

\begin{remark}
The commutativity of $T^{r,d}$ is not necessary for the definition, since it can be deduced from the remaining conditions \cite{o.f}.
\end{remark}

We consider the following group homomorphism:
$$\rho: \mathbb{G}^r_m\to PGL(r,\mathbb{C}),$$
given by
$$\rho(\lambda_1,...,\lambda_r)=
\begin{bmatrix}
    1 &   &     &0    \\
      &\lambda_1 &   &   \\
      &   & \ddots &   \\
       0&    &  & \lambda_r
\end{bmatrix}
$$
Let $P^{r,d}=T^{r,d}\times_{\rho}\mathbb{P}^r=T^{r,d}\times\mathbb{P}^r/\mathbb{G}^r_m$, which is a $\mathbb{P}^r$-bundle over $A^d$. We can view $P^{r,d}$ as a compactification of $T^{r,d}$ which naturally carries the $T^{r,d}$ action on it and denote the boundary divisor $L$, which has simple normal crossings and is $T^{r,d}$-invariant for each stratum.

Abusing the notations a little bit, over $T^{r,d}$, we have the morphism $[k]:T^{r,d}\to T^{r,d}$ given by multiplication by $k$, if we use addition for the group operation on $T^{r,d}$. We can also define $[k]:\mathbb{P}^r\to \mathbb{P}^r$, given by the Fermat's morphism $[z_0,...,z_r]\to [z^k_0,...,z^k_r]$. It is not hard to check that, combining these two morphisms, we get a finite morphism
$$[k]:P^{r,d}\to P^{r,d},$$
which is étale over $T^{r,d}$, and ramified along the boundary divisor $L$ by degree $k$.

\begin{prop}
The logarithmic cotangent bundle $\Omega^1_{P^{r,d}}(log\ L)$ over $P^{r,d}$ is a globally free vector bundle. In particular, we have $\omega_{P^{r,d}}(L)\simeq \mathcal{O}_{P^{r,d}}$.
\end{prop}

\begin{proof}
Since $T^{r,d}$ is a symmetric space, the global $T^{r,d}$-invariant $1$-forms generate the cotangent space $\Omega^1_{T^{r,d}}$ of $T^{r,d}$. This free vector space can be extended onto $P^{r,d}$ viewed as a sub-bundle of $\Omega^1_{P^{r,d}}(*L)$. Now we need to show that this sub-vector bundle, denoted as $\mathcal{T}$, is $\Omega^1_{P^{r,d}}(log\ L)$.

To see that $\mathcal{T}\subset \Omega^1_{P^{r,d}}(log\ L)$, pick any local section of $\Omega^1_{P^{r,d}}(*L)$, which is $T^{r,d}$-invariant, we only need to show that it has at most a simple log pole along $L$. Restricted onto each $\mathbb{P}^r$-fiber of $P^{r,d}$, its non-holomorphic part is a linear combination of $\frac{dx_i}{x_i}$. Hence locally on $P^{r,d}$, its non-holomorphic part is a linear combination of $\frac{df_i}{f_i}$, where $f_i$ are local functions which define $L$. For $\Omega^1_{P^{r,d}}(log\ L)\subset\mathcal{T}$, we only need to show that  $\frac{df_i}{f_i}$ is a local section of $\mathcal{T}$.
Since we know that $L$ is $T^{r,d}$-invariant for each irreducible component, $\frac{df_i}{f_i}$ is $T^{r,d}$-invariant. 
\end{proof}

For any smooth quasi-projective variety $V$, we use the notation $T_1(V)$ to denote the space of global holomorphic one-forms on $V$ that can be extended to a log-one-form onto a smooth compactification of $V$. This space does not depend on the smooth compactification, \cite{s.i}. Further, we can canonically define a quasi-albanese morphism $a_V:V\to \tilde{A}_V$, where $\tilde{A}_V$ is a quasi-abelian variety and we call it the quasi-albanese variety of $V$. In particular, if $V$ is projective, $a_V$ is just the usual albanese morphism.
\begin{prop}\label{prop:globallog}
We have that 
$$a_V^*:T_1(\tilde{A}_V)\to T_1(V)$$
is an isomorphism.
\end{prop}
We refer to \cite[3.12]{o.f} for the proof.
\section{Proof of Theorem 5}
In this section, we apply the Higgs bundle construction from \cite{ps3}, and use the first reduction step in \cite{ps2}, which also appeared in \cite{vz}, to prove Theorem 5

\begin{proof}[Proof of Theorem 5]
Assume that the morphism $f:X\setminus D\to T^{r,d}$ is smooth. After blowing up $X$ on its boundary $D$ and taking a log-resolution if necessary, we can extend $f$ to $X\to P^{r,d}$ (\cite[II, Example 7.17.3]{r.h}), which we also denote by $f: X\to P^{r,d}$. Since the logarithmic Kodaira dimension does not depend the smooth compactification, we still have that $(X,D)$ is of log-general type. Denote the boundary divisor on $P^{r,d}$ by $L$, which has $r+1$ $T^{r,d}$-invariant components and consists $r+1$ general hyperplanes restricting on each $\mathbb{P}^r$ fiber.

Denote $p:P^{r,d}\to A$ the natural projection. Pick any ample effective divisor $E$ on $A$ and let $\mathcal{G}=\mathcal{O}_{P^{r,d}}(L)\otimes p^*\mathcal{O}_A(E)$ be a line bundle on $P^{r,d}$. Since $\omega_X(D)$ is big, there exists a positive integer $k$ such that 
$$H^0(X, (\omega_X(D))^{k}\otimes f^*\mathcal{G}^{-1})\neq 0.$$

Let $[2k]: P^{r,d}\to P^{r,d}$ be the finite morphism which is defined in Section 3. Let $\phi:X'\to X$ be the corresponding base change followed by normalization and a log-resolution, and let $(X',D')$ be the induced log-smooth pair. We may assume that $X'\setminus D'=X\setminus D\times_{T^{r,d}}T^{r,d}$. Let $f':X'\to P^{r,d}$ be the induced morphism.
\begin{center}
\begin{tikzcd}
X' \arrow{r}{\phi} \arrow{d}{f'}
&X \arrow{d}{f}\\
P^{r,d}\arrow{r}{[2k]}
&P^{r,d}
\end{tikzcd}
\end{center}
Since $[2k]$ is étale over $T^{r,d}$, we have that $f'$ is smooth over $T^{r,d}$.
Since $D'\subset f'^*p^* L$, $\omega_{P^{r,d}}\simeq \mathcal{O}_{P^{r,d}}(-L)$ and consider \cite[II, Proposition 5.20]{j.k}, then $\phi^*\omega_X(D)\subset \omega_{X'}(D')\subset \omega_{X'/P^{r,d}} $. Further, $[2k]^*L=2kL$ and $[2k]^*E-2kE$ is linear equivalent to an effective divisor, \cite[II, 6]{d.m}. Hence
$$
\phi^*( \omega_{X}(D)^{k}\otimes f^*\mathcal{G}^{-1})\subset  \omega_{X'/P^{r,d}}^{k}\otimes f'^*(\mathcal{G}^{-2k}).
$$
Let $\mathcal{B}=\omega_{X'/P^{r,d}}\otimes f'^*(\mathcal{G}^{-2})$. Then, by the above inclusion, we see that 
$$H^0(X',\mathcal{B}^k)\neq 0.$$ 

Now, let's apply  \cite[Theorem 9.4 (a)]{ps3} to our situation. Note that even through $f'$ does not have connected fibers, \cite[Theorem 9.4]{ps3} still applies with minor changes. (See the remark below.) In our case, we can let $D_{f'}=L$, which, by assumption, contains the image of all singular fibers. Hence 
$$\mathcal{F}_0\supset \mathcal{G}^2(-D_{f'})\supset \mathcal{G} $$
Since $\mathcal{G}$ is big, by \cite[Theorem 9.4 (c) and Theorem 19.1]{ps3}, we conclude that $\omega_{P^{r,d}}(D_{f'})\simeq \mathcal{O}_{P^{r,d}}$ is big which is absurd.
\end{proof}
\begin{remark}
In \cite[Theorem 9.4]{ps3}, if we drop the assumption that $f$ has connected fiber, we only lose the property that $f_*\mathcal{O}_Y=\mathcal{O}_X$. The proof still works except that we must replace the result (a) by: One has $L(-D_f)\otimes f_*\mathcal{O}_Y\subset \mathcal{F}_0\subset L\otimes f_*\mathcal{O}_Y$. As long as $f$ is surjective, we still have $f_*\mathcal{O}_Y\supset\mathcal{O}_X$. Hence $L(-D_f)\subset \mathcal{F}_0$ which is what is needed in the proof.
\end{remark}

Finally, we show that Conjecture 3 implies Theorem 5 due to the following

\begin{lemma}\label{L:g}
Given a morphism $f: X\to P^{r,d}$, where $X$ is smooth and quasi-projective and $P^{r,d}$ is the $\mathbb{P}^r$-bundle compactification of a quasi-abelian variety $T^{r,d}$, with $L$ be the boundary divisor on $P^{r,d}$. Let $D=(f^*L)_{red}$ and assume that $(X,D)$ is log-smooth. Then for general $\theta \in H^0(\Omega^1_{P^{r,d}}(log\  L))$, $f^* \theta\in H^0 (\Omega^1_X(log\ D))$ has a simple log pole on every point of $D$. In particular, it does not vanish at any point of $D$.
\end{lemma}

\begin{proof}
Since we can cover $P^{r,d}$ by finite trivial $\mathbb{P}^r$-bundles, we only need to prove the lemma over each of them. However, to show the trivial bundle case, it is not hard to see that we only need to show the case that $P^{r,d}=\mathbb{P}^r$, and $L=L_0+...+L_{r}$ is $r+1$ hyperplanes of general position.

After changing coordinates, we can assume that $L_i$ is defined by $x_i=0$ and $U_i:=\mathbb{P}^r\setminus L_i$ for $i=0,1,...,r$.

After giving a finite affine covering $\{X_{i,j}\}$ of $X_i= f^{-1}(U_i)$, we only need to show the claim is true for each $f|_{X_{i,j}}: X_{i,j}\to U_i$.

Without loss of generality, we pick $U_0$ which is defined by $x_0\neq 0$, hence 
$$H^0(\mathbb{P}^r, \Omega_{\mathbb{P}^r}^1(log\  L))|_{U_0}$$
is a $\mathbb{C}$-vector space with a base $\{dx_i/x_i\}$, for $1\leq i\leq r$. Pick one of those $X_{0,j}$, denoting it by $V$, and denote $f|_{X_{0,j}}$ by $g:V \to U_0$. Abusing the notation a little bit, we let $L_i$ be the hyperplane defined by $x_i=0$ on $U_0$ and denote $g^*(L_i)=\sum_j m_{i,j}D_{i,j}$ with $m_{i,j}>0$. Set $x_{i,j}$ as the function on $V$ that defines $D_{i,j}$. Then, by a local computation, we have
$$g^*x_i=u_i\prod_{(i,j)}x_{i,j}^{m_{i,j}},$$
where $u_i$ is a nowhere vanishing function on $V$. Hence, $g^*(dx_i/x_i)$ over any $x\in g^{-1}(L_i)$ is of the form 
$$\sum_{(i,j)| x\in  D_{i,j}} m_{i,j}\frac{dx_{i,j}}{x_{i,j}}+\text{holomorphic part}.$$
Hence, it is clear that, for a general linear combination of ${dx_i/x_i}$, its pull back will have a pole on every point of $g^{-1}(L_1+...+L_r)$ and the coefficients of the poles do not have zero locus on that divisor.
\end{proof}

\end{document}